\documentclass[10pt]{article}
\usepackage{latexsym}
\usepackage{amssymb}
\usepackage{amsmath}
\usepackage{amsthm}
\usepackage{thmtools}
\usepackage{cite}
\usepackage[all]{xy}
\usepackage{framed}
\usepackage{hyperref}
\usepackage{graphicx}
\usepackage{parskip}
\usepackage{hyperref}

\begingroup
    \makeatletter
    \@for\theoremstyle:=definition,remark,plain\do{%
        \expandafter\g@addto@macro\csname th@\theoremstyle\endcsname{%
            \addtolength\thm@preskip\parskip
            }%
        }
\endgroup

\makeatletter
\xydef@\txt@ii#1{\vbox{\vspace*{-5pt}%
 \let\\=\cr
 \tabskip=\z@skip \halign{\relax\hfil\txtline@@{##}\hfil\cr\leavevmode#1\crcr}}}
\makeatother

\theoremstyle{definition}
\newtheorem{thm}{Theorem}[section]
\newtheorem{lem}[thm]{Lemma}
\newtheorem{cor}[thm]{Corollary}
\newtheorem{defn}[thm]{Definition}
\newtheorem{propn}[thm]{Proposition}
\newtheorem*{thm*}{Theorem}

\newtheorem*{nts}{Note to self}

\theoremstyle{remark}
\newtheorem*{rk}{Remark}

\newtheoremstyle{custthm}{\parskip}{}{\normalfont}{}{\bfseries}{.}{ }{\thmname{#1} \thmnote{#3}}
\theoremstyle{custthm}
\newtheorem*{letterthm}{Theorem}

\newcommand{\nio}{\mathrm{nio}}

\newcommand{\tensor}[1]{\underset{#1}{\otimes}}

\newcommand{\nopen}{\underset{O}{\lhd}}

\newcommand{\Sat}{\mathrm{Sat}\,}

\newcommand{\Aut}{\mathrm{Aut}}

\newcommand{\fn}{\mathbf{FN}_p}
\newcommand{\Rinfty}{\mathbb{R}\cup\{\infty\}}
\newcommand{\actson}{\mathbin{\text{\rotatebox[origin=c]{-90}{$\circlearrowright$}}}}

{%
\begin{framed}\begin{nts}%
}%
{%
\end{nts}\end{framed}%
}

\begin{document}

\binoppenalty=\maxdimen
\relpenalty=\maxdimen

\title{On the structure of virtually nilpotent compact $p$-adic analytic groups}
\author{Billy Woods}
\date{\today}
\maketitle
\begin{abstract}
Let $G$ be a compact $p$-adic analytic group. We recall the well-under\-stood finite radical $\Delta^+$ and FC-centre $\Delta$, and introduce a $p$-adic analogue of Roseblade's subgroup $\nio(G)$, the unique largest orbitally sound open normal subgroup of $G$. Further, when $G$ is nilpotent-by-finite, we introduce the finite-by-(nilpotent $p$-valuable) radical $\fn(G)$, an open characteristic subgroup of $G$ contained in $\nio(G)$. By relating the already well-known theory of isolators with Lazard's notion of $p$-saturations, we introduce the isolated lower central (resp. isolated derived) series of a nilpotent (resp. soluble) $p$-valuable group of finite rank, and use this to study the conjugation action of $\nio(G)$ on $\fn(G)$. We emerge with a structure theorem for $G$,
$$1 \leq \Delta^+ \leq \Delta \leq \fn(G) \leq \nio(G) \leq G,$$
in which the various quotients of this series of groups are well understood. This sheds light on the ideal structure of the Iwasawa algebras (i.e. the completed group rings $kG$) of such groups, and will be used in future work to study the prime ideals of these rings.
\end{abstract}

\newpage
\tableofcontents

%
%
%
%
%
%

\newpage

\section*{Introduction}

We aim to study the structure of certain compact $p$-adic analytic groups $G$. This will crucially underpin later work in which we will explore the ring-theoretic properties of completed group rings $kG$, where $k$ is a finite field of characteristic $p$ and $G$ is a nilpotent-by-finite compact $p$-adic analytic group.

There is much in common between the theory of polycyclic-by-finite groups (and their group rings) and the theory of polyprocyclic-by-finite compact $p$-adic analytic groups (and their completed group rings). See \cite{DDMS} and \cite{lazard}, or the more recent survey paper by Ardakov and Brown \cite{ardakovbrown}, for an overview of the latter.

It is known that

\centerline{
\xymatrix@1{
\left\{ \txt{uniform groups \\ \cite[Definition 4.1]{DDMS}} \right\} \ar@{}[r]|-*[@]{\subseteq}&
\left\{ \txt{$p$-valuable groups \\ \cite[III, 2.1.2]{lazard} } \right\} \ar@{}[r]|-*[@]{\subseteq}&
\left\{ \txt{compact $p$-adic \\ analytic groups \\ \cite[Definition 8.14]{DDMS} } \right\}, 
}
}

with the first inclusion coming from \cite[Definition 1.15; notes at end of chapter 4]{DDMS}, and the second from \cite[Corollary 8.34]{DDMS}.

Note also that compact $p$-adic analytic groups $G$ are profinite groups satisfying \textbf{Max}: every nonempty set of closed subgroups of $G$ contains a maximal element. Indeed, in the case when $G$ is $p$-valuable, this follows from \cite[III, 3.1.7.5]{lazard}; in the general case, $G$ contains a uniform open normal subgroup $U$ by \cite[8.34]{DDMS}. Indeed, \cite[8.34]{DDMS} implies that compact $p$-adic analytic groups are precisely extensions of uniform (or $p$-valuable) groups by finite groups.

We aim eventually to extend some of the work of Roseblade \cite{roseblade} and Letzter and Lorenz \cite{ll} to the domain of compact $p$-adic analytic groups, building on work by Ardakov \cite{ardakovInv}. Our main results are as follows.

Let $G$ be a compact $p$-adic analytic group, and $H$ a closed subgroup. Following Roseblade \cite{roseblade}, we will say that $H$ is \emph{orbital} (or \emph{$G$-orbital}) if it only has finitely many $G$-conjugates, or equivalently if its normaliser $\mathbf{N}_G(H)$ is open in $G$; and $H$ is \emph{isolated orbital} (or \emph{$G$-isolated orbital}) if $H$ is orbital, and given any other closed orbital subgroup $H'$ of $G$ with $H \lneq H'$, we have $[H':H] = \infty$. $G$ is then said to be \emph{orbitally sound} if all its isolated orbital closed subgroups are in fact normal.

We define the Roseblade subgroup
$$\nio(G) = \bigcap_H \mathbf{N}_G(H),$$
where this intersection is taken over all isolated orbital closed subgroups $H$ of $G$. In section 2, we prove:

\begin{letterthm}[A]
Let $G$ be a compact $p$-adic analytic group. Then $\nio(G)$ is an orbitally sound, open, characteristic subgroup of $G$, and contains all finite-by-nilpotent closed normal subgroups of $G$.\qed
\end{letterthm}

Now let $G$ be a $p$-valuable group. Ardakov shows, in \cite[Lemma 8.4(a)]{ardakovInv}, that the centre $Z(G)$ is isolated orbital in $G$; we may then deduce from Lemma \ref{lem: properties of orbitally sound groups}(i) that the usual upper central series
$$1 \leq Z(G) \leq Z_2(G) \leq \dots$$
of $G$ as an abstract group consists of \emph{isolated orbital} subgroups of $G$, so that any $p$-valuation on $G$ naturally induces $p$-valuations on each $Z_i(G)$ and $G/Z_i(G)$. Unfortunately, in general, the (abstract) lower central series $\{\gamma_i\}$ (or derived series $\{\mathcal{D}_i\}$) of $G$ will not necessarily consist of \emph{isolated} orbital subgroups, so $G/\gamma_i$ (or $G/\mathcal{D}_i$) will not necessarily remain $p$-valuable.

In section 3, we introduce appropriate ``isolated" analogues of these series for $p$-valuable groups:

\begin{letterthm}[B]
Let $G$ be a $p$-valuable group. Then there exists a unique fastest descending series of isolated orbital closed normal subgroups of $G$, the \emph{isolated lower central series},
$$ G = G_1 \rhd G_2 \rhd \dots , $$
with the properties that each $G_i$ is characteristic in $G$, $G_i/G_{i+1}$ is abelian for each $i$, $[G_i, G_j] \leq G_{i+j}$ for all $i$ and $j$, and there exists some $r$ with $G_r = 1$ if and only if $G$ is nilpotent.

There exists also a unique fastest descending series of isolated orbital closed normal subgroups of $G$, the \emph{isolated derived series},
$$ G = G^{(0)} \rhd G^{(1)} \rhd \dots, $$
with the properties that each $G^{(i)}$ is characteristic in $G$, $G^{(i)}/G^{(i+1)}$ is abelian for each $i$, and there exists some $r$ with $G^{(r)} = 1$ if and only if $G$ is soluble.\qed
\end{letterthm}

Now let $G$ be a fixed compact $p$-adic analytic group. Define the two closed subgroups

\begin{align*}
\Delta^{\hphantom{+}} &= \big\{x\in G \;\big|\; [G: \mathbf{C}_G(x)] < \infty\big\},\\
\Delta^+ &= \big\{x\in \Delta \;\big|\; o(x) < \infty\big\},
\end{align*}

where $o(x)$ denotes the order of $x$. We will show that we always have an inclusion of subgroups
$$1\leq \Delta^+\leq \Delta\leq \nio(G)\leq G.$$

Suppose further that $G$ is a \emph{nilpotent-by-finite} compact $p$-adic analytic group. Consider the set of finite-by-(nilpotent $p$-valuable) open normal subgroups $H$ of $G$ -- that is, the set of open normal subgroups $H$ that contain a finite normal subgroup $F\lhd H$, such that $H/F$ is nilpotent and $p$-valuable. The main result of section 5 is:

\begin{letterthm}[C]
This set contains a unique maximal element $H$, which is characteristic in $G$. There is an inclusion of subgroups
$$\Delta \leq H \leq \nio(G),$$
and the quotient group $\nio(G)/H$ is isomorphic to a subgroup of the group of torsion units of $\mathbb{Z}_p$.\qed
\end{letterthm}

We will denote this $H$ by $\fn(G)$, the \emph{finite-by-(nilpotent $p$-valuable) radical} of $G$.

Putting this together with the technical results of section 4, we can now understand the conjugation action of $\nio(G)$ on $\fn(G)$:

\begin{letterthm}[D]
Let $G$ be a nilpotent-by-finite compact $p$-adic analytic group, and $H = \fn(G)$ its finite-by-(nilpotent $p$-valuable) radical. Write $$N = N_1 \rhd N_2 \rhd \dots \rhd N_r = 1$$ for the isolated lower central series of $N = H/\Delta^+$, and let $H_i$ be the full preimage in $G$ of $N_i$. Then $\nio(G)/H$ is isomorphic to a subgroup of $t(\mathbb{Z}_p^\times)$, the group of torsion units of $\mathbb{Z}_p$, and so is cyclic; let $a$ be a preimage in $\nio(G)$ of a generator of $\nio(G)/H$. Then conjugation by $a$ acts on each $H_i$, and hence induces an action on the (free, finite-rank) $\mathbb{Z}_p$-modules $H_i/H_{i+1} = A_i$. In multiplicative notation, there is some scalar $\zeta\in t(\mathbb{Z}_p^\times)$ such that $x^a = x^{\zeta^i}$ for all $x\in A_i$.\qed
\end{letterthm}

Putting these ingredients together allows us to understand the structure of a nilpotent-by-finite compact $p$-adic analytic group $G$, which, for convenience, we display in the following diagram.

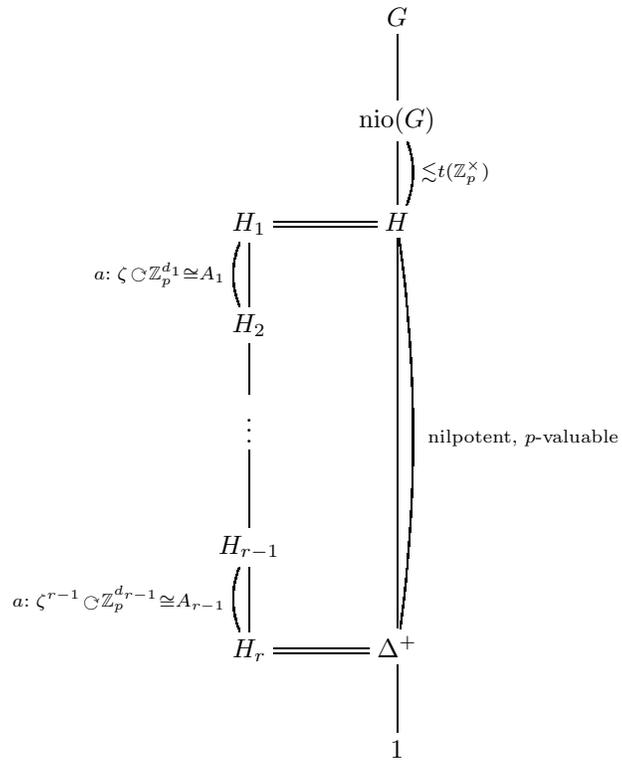
\centerline{
\xymatrix{
& G\ar@{-}[d]\\
& \nio(G)\ar@{-}[d]\ar@{-}@/^/[d]^{\lesssim t(\mathbb{Z}_p^\times)}\\
H_1\ar@{=}[r]\ar@{-}[d]\ar@{-}@/_/_{a:\; \zeta \, \actson \, \mathbb{Z}_p^{d_1} \cong A_1}[d] & H\ar@{-}[dddd]\ar@{-}@/^/[dddd]^{\text{ nilpotent, } p\text{-valuable}}\\
H_2\ar@{-}[d] & \\
\vdots\ar@{-}[d] & \\
H_{r-1}\ar@{-}[d]\ar@{-}@/_/_{a:\; \zeta^{r-1} \, \actson \, \mathbb{Z}_p^{d_{r-1}} \cong A_{r-1}}[d] & \\
H_r\ar@{=}[r]& \Delta^+\ar@{-}[d]\\
& 1
}
}

\newpage
\section{Preliminaries}

\begin{defn}
Let $G$ be a profinite group. A closed subgroup $H$ of $G$ is $G$-\emph{orbital} (or just \emph{orbital}, when the group $G$ is clear from context) if $H$ has only finitely many $G$-conjugates, i.e. if $\mathbf{N}_G(H)$ is open in $G$. Similarly, an element $x\in G$ is \emph{orbital} if $[G: \mathbf{C}_G(x)] < \infty$.
\end{defn}

\begin{rk}
Note that, if $G$ is compact $p$-adic analytic, it is profinite (by \cite[8.34]{DDMS}).
\end{rk}

\begin{defn}\label{defn: Delta}
The \emph{FC-centre} $\Delta(G)$ of an arbitrary group $G$ is the subgroup of all orbital elements of $G$. The \emph{finite radical} $\Delta^+(G)$ of $G$ is the subgroup of all torsion orbital elements of $G$.
\end{defn}

\begin{rk}
Throughout this paper, we will write as shorthand $\Delta^+ = \Delta^+(G)$ and $\Delta = \Delta(G)$. Also throughout this paper, all subgroups will be closed, all homomorphisms continuous, etc. unless otherwise specified.
\end{rk}

\begin{lem}\label{lem: Delta and Delta+}
Let $G$ be a compact $p$-adic analytic group. For convenience, we record a few basic properties of $\Delta$ and $\Delta^+$.

\begin{itemize}
\item[(i)] $\Delta^+$ is finite.
\item[(ii)] If $H$ is an open subgroup of $G$, then $\Delta^+(H) \leq \Delta^+(G)$ and $\Delta(H) \leq \Delta(G)$.
\item[(iii)] When $G$ is compact $p$-adic analytic, $\Delta^+$ and $\Delta$ are closed in $G$.
\item[(iv)] $\Delta^+$ and $\Delta$ are characteristic subgroups of $G$.
\item[(v)] $\Delta/\Delta^+$ is a torsion-free abelian group.
\end{itemize}
\end{lem}

\begin{proof}
$ $

\begin{itemize}
\item[(i)] $\Delta^+$ is generated by the finite normal subgroups of $G$ \cite[5.1(iii)]{passmanICP}. It is obvious that the compositum of two finite normal subgroups is again finite and normal. Now suppose that $\Delta^+$ is infinite, and take an open uniform subgroup $H$ of $G$ \cite[4.3]{DDMS}: then $\Delta^+\cap H$ is non-trivial, and so we must have some finite normal subgroup $F$ with $F\cap H$ non-trivial. But $F$ is torsion, so this contradicts the fact that $H$ is torsion-free \cite[4.5]{DDMS}.
\item[(ii)] If an element $x\in H$ has finitely many $H$-conjugates, and $H$ has finite index in $G$, then $x$ has finitely many $G$-conjugates.
\item[(iii)] $\Delta^+$ is closed because it is finite.

For the case of $\Delta$, suppose first that $G$ is $p$-valued \cite[III, 2.1.2]{lazard}. Now, any orbital $x\in G$ has $\mathbf{C}_G(x)$ open in $G$, and so, for any $g\in G$, there exists some $n$ with $g^{p^n}\in \mathbf{C}_G(x)$, i.e. $g^{p^n}x = xg^{p^n}$. This implies that $(g^x)^{p^n} = g^{p^n}$, and so by \cite[III, 2.1.4]{lazard}, we get $g^x = g$. Hence $\mathbf{C}_G(x) = G$. In other words, $\Delta = Z(G)$, which is closed in $G$.

When $G$ is not $p$-valued, it still has an open $p$-valued subgroup $N$ \cite[4.3]{DDMS}. Clearly $\Delta(N) = \Delta(G) \cap N$, and so $[\Delta(G) : \Delta(N)] \leq [G:N] < \infty$. So $\Delta(G)$ is a finite union of translates of $Z(N)$, which is closed in $N$ and hence closed in $G$.
\item[(iv)] See \cite[discussion after lemma 4.1.2 and lemma 4.1.6]{passmanASGR}.
\item[(v)] See \cite[lemma 4.1.6]{passmanASGR}.\qedhere
\end{itemize}
\end{proof}

\textbf{Throughout the remainder of this subsection, $G$ is a profinite group unless stated otherwise.}

\begin{defn}\label{defn: orbitally sound}
An orbital closed subgroup $H$ of $G$ is \emph{isolated} if, for all orbital closed subgroups $H'$ of $G$ with $H \lneq H' \leq G$, we have $[H':H] = \infty$. (We will sometimes say that a closed subgroup is \emph{$G$-isolated orbital} as shorthand for \emph{isolated as an orbital closed subgroup of $G$}.) Following Passman \cite[definition 19.1]{passmanICP}, if all isolated orbital closed subgroups of $G$ are in fact normal, we shall say that $G$ is \emph{orbitally sound}.
\end{defn}

We record a few basic properties, before showing that this definition is the same as the one given in \cite[1.3]{roseblade} and \cite[5.8]{ardakovInv} (in Lemma \ref{lem: defns of orbitally sound are equivalent} below).

\begin{lem}\label{lem: properties of orbitally sound groups}
Let $N$ be a closed normal subgroup of $G$.

\begin{itemize}
\item[(i)] Suppose $H$ is a closed subgroup of $G$ containing $N$. Then $H/N$ is $(G/N)$-orbital if and only if $H$ is $G$-orbital; and $H/N$ is $(G/N)$-isolated orbital if and only if $H$ is $G$-isolated orbital.
\item[(ii)] Suppose $G$ is orbitally sound. Then $G/N$ is orbitally sound.
\item[(iii)] Suppose $N$ is finite and $G/N$ is orbitally sound. Then $G$ is orbitally sound.
\end{itemize}
\end{lem}

\begin{proof}
$ $

\begin{itemize}
\item[(i)] It is easily checked that $\mathbf{N}_{G/N}(H/N) = \mathbf{N}_G(H) / N$, and so
\begin{equation*}
[G: \mathbf{N}_G(H)] = [G/N: \mathbf{N}_G(H)/N] = [G/N: \mathbf{N}_{G/N}(H/N)].
\end{equation*}
So $H$ is orbital if and only if $H/N$ is orbital. Suppose these two groups are both orbital, and let $H'$ be an orbital closed subgroup of $G$ with $H \lneq H' \leq G$: then $[H':H] = [H'/N : H/N]$, so $H$ is isolated if and only if $H'$ is isolated.
\item[(ii)] Let $H/N$ be an isolated orbital closed subgroup of $G/N$. Then, by (i), $H$ is an isolated orbital subgroup of $G$, so $H\lhd G$, and so $H/N \lhd G/N$.
\item[(iii)] Let $H$ be an isolated orbital closed subgroup of $G$, and $H'$ an orbital closed subgroup of $G$ with $H \lneq H' \leq G$. If $H$ contains $N$, then we may apply (i) to show that $[H':H] = [H'/N : H/N] = \infty$ as $G/N$ is orbitally sound. But $H$ must contain $N$: indeed, as $H$ is $G$-orbital, clearly $HN/N$ is $G/N$-orbital, and so, by (i), $HN$ is $G$-orbital. But $N$ is finite, so $[HN:H] < \infty$, and $H$ is isolated, so $H = HN$.\qedhere
\end{itemize}
\end{proof}

\textbf{From now on, we assume that $G$ is a profinite group satisfying the \emph{maximum condition} on closed subgroups}: every nonempty set of closed subgroups of $G$ has a maximal element.

\begin{rk}
Note that, if $G$ is compact $p$-adic analytic, it satisfies the maximum condition on closed subgroups. Indeed, this is true for $p$-valuable $G$ by \cite[III, 3.1.7.5]{lazard}, and hence true for any compact $p$-adic analytic group $G$, as $G$ contains a uniform (hence $p$-valuable) subgroup of finite index \cite[8.34]{DDMS}.
\end{rk}

\begin{defn}\label{defn: isolator}
If $H$ is an orbital closed subgroup of $G$, we define its \emph{isolator} $\mathrm{i}_G(H)$ in $G$ to be the closed subgroup of $G$ generated by all orbital closed subgroups $L$ of $G$ containing $H$ as an open subgroup, i.e. with $[L:H] < \infty$.

Once we have proved that $\mathrm{i}_G(H)$ is indeed an isolated orbital closed subgroup of $G$ containing $H$ as an open subgroup, it will be clear from the definition that it is the unique such closed subgroup.
\end{defn}

We now prove some basic properties of $\mathrm{i}_G(H)$, following \cite{passmanICP}.

\begin{propn}\label{propn: orbital subgroups are open in their isolators}
Suppose $H$ is an orbital closed subgroup of $G$. Then $H$ is open in $\mathrm{i}_G(H)$.
\end{propn}

\begin{proof}
We first show that, if $L_1$ and $L_2$ are orbital subgroups of $G$ containing $H$ as an open subgroup, then $[\langle L_1, L_2\rangle : H] < \infty$. Write $(-)^\circ$ for $\bigcap_{g\in G}(-)^g$. Suppose without loss of generality that $G = \langle L_1, L_2\rangle$, and that $H^\circ = 1$ (by passing to $G/H^\circ$).

For $i = 1, 2$, as $[L_i : H] < \infty$ and as $H, L_i$ are all orbital, we may take an open normal subgroup $N$ of $G$ such that $[N, L_i] \subseteq H$. Indeed, $\mathbf{N}_G(L_i)$ is a subgroup of finite index in $G$, and permutes the (finitely many) left cosets of $H$ in $L_i$ by left multiplication; take $N_i$ to be the kernel of this action, and set $N = N_1 \cap N_2$.

Hence $[N\cap H, L_i] \subseteq N\cap H$, i.e. both $L_1$ and $L_2$ normalise $N\cap H$, so $G$ normalises $N\cap H$. So $N\cap H$ is a normal subgroup of $G$ contained in $H$, and by assumption must be trivial. But $N$ was an open subgroup of $G$, so $H$ must have been finite, and so $L_1$ and $L_2$ must be \emph{finite} orbital subgroups of $G$. This implies that $L_i \leq \Delta^+$, and hence $G = \Delta^+$, so that $G$ is finite, as required.

Now, in the general case, $G$ satisfies the maximal condition on closed subgroups, so we can choose $L$ maximal subject to $L$ being orbital and $[L:H] < \infty$. This $L$ is $\mathrm{i}_G(H)$ and contains $H$ as an open subgroup.
\end{proof}

\begin{lem}\label{lem: properties of isolator}
$ $

\begin{itemize}
\item[(i)] Suppose $H$ is an orbital closed subgroup of $G$. Then $\mathrm{i}_G(H)$ is an isolated orbital closed subgroup of $G$. Furthermore, if $H$ is normal in $G$, then so is $\mathrm{i}_G(H)$.
\item[(ii)] Suppose $G$ is orbitally sound and $H$ is a closed subgroup of finite index. Then $H$ is orbitally sound.
\end{itemize}
\end{lem}

\begin{proof}
$ $

\begin{itemize}
\item[(i)] If $\mathrm{i}_G(H)$ is orbital, then by Proposition \ref{propn: orbital subgroups are open in their isolators}, it is isolated (by construction). But $H$ has finite index in $\mathrm{i}_G(H)$, so $\mathrm{i}_G(H)$ must be generated by a finite number of subgroups $L_1, \dots, L_n$ containing $H$ as a subgroup of finite index. So
$$\bigcap_{i=1}^n \mathbf{N}_G(L_i) \leq \mathbf{N}_G(\mathrm{i}_G(H)),$$
and as each $\mathbf{N}_G(L_i)$ is open in $G$, so is $\mathbf{N}_G(\mathrm{i}_G(H))$.

Now suppose that $H$ is normal in $G$. To see that $\mathrm{i}_G(H)$ is normal in $G$, fix $g\in G$, and note that conjugation by $g$ fixes $H$ and therefore simply permutes the set of orbital closed subgroups $L$ of $G$ containing $H$ as an open subgroup, i.e. permutes the set of subgroups of $G$ that generate $\mathrm{i}_G(H)$ (see Definition \ref{defn: isolator}).
\item[(ii)] Let $K$ be an isolated $H$-orbital closed subgroup of $H$. Then $K$ is $G$-orbital, so $\mathrm{i}_G(K)$ is an isolated orbital subgroup of $G$, and so is normal in $G$. Hence $\mathrm{i}_G(K) \cap H$ is normal in $H$. But $[\mathrm{i}_G(K) : K] < \infty$, so $[\mathrm{i}_G(K)\cap H : K] < \infty$, and hence $\mathrm{i}_G(K)\cap H = K$, as $K$ was assumed to be isolated in $H$.\qedhere
\end{itemize}
\end{proof}

\begin{lem}\label{lem: 1-1 corresp between isolated orbitals}
Let $H$ be an open normal subgroup of $G$. Then there is a one-to-one correspondence

\centerline{
\xymatrix@R-2pc{
\left\{ \txt{isolated orbital \\ closed subgroups of $G$} \right\} \ar@{<->}[r]&
\left\{ \txt{isolated orbital \\ closed subgroups of $H$} \right\}\\
G'\ar@{|->}[r]& \mathrm{i}_H(G'\cap H),\\
\mathrm{i}_G(H')\ar@{<-|}[r]& H'.
}
}
\end{lem}

\begin{proof}
Suppose first that $H'$ is an arbitrary orbital closed subgroup of $H$. That is, $\mathbf{N}_H(H')$ is open in $H$, hence also in $G$, and so $\mathbf{N}_G(H')$ must be open in $G$. Therefore $H'$ is also $G$-orbital, and so, by Lemma \ref{lem: properties of isolator}(i), $\mathrm{i}_G(H')$ is an isolated orbital closed subgroup of $G$.

Conversely, take a $G$-isolated orbital closed subgroup $G'$: then $\mathbf{N}_G(G')$ is open in $G$, as $G'$ is $G$-orbital, which implies that $\mathbf{N}_H(G'\cap H)$ is open in $H$, i.e. that $G'\cap H$ is $H$-orbital. Now $\mathrm{i}_H(G'\cap H)$ is $H$-isolated orbital by Lemma \ref{lem: properties of isolator}(i).

Now we will show that these correspondences are mutually inverse.

In one direction, we must take $H'$ as above, assume further that it is $H$-isolated orbital, and show that $\mathrm{i}_H(\mathrm{i}_G(H')\cap H) = H'$. To do this, note that $\mathrm{i}_G(H')$ contains $H'$ as an open subgroup by Proposition \ref{propn: orbital subgroups are open in their isolators}, so $\mathrm{i}_G(H')\cap H$ is an $H$-orbital closed subgroup (by the correspondence above) containing the $H$-isolated orbital $H'$ as an open subgroup, and so by definition the two must be equal.

For the converse direction, we must take $G'$ as above and show that $\mathrm{i}_G(\mathrm{i}_H(G'\cap H))$ is equal to $G'$. But clearly $\mathrm{i}_G \mathrm{i}_H = \mathrm{i}_G$, and both $G'$ and $\mathrm{i}_G(G'\cap H)$ are $G$-isolated orbital ($G'$ by assumption, $\mathrm{i}_G(G'\cap H)$ by definition) and contain $G'\cap H$ as an open subgroup, so by uniqueness (see Definition \ref{defn: isolator}), they are equal.
\end{proof}

\begin{lem}\label{lem: defns of orbitally sound are equivalent}
The following are equivalent:

\begin{itemize}
\item[(i)]
Any isolated orbital closed subgroup $H$ of $G$ is normal.
\item[(ii)]
Any orbital closed subgroup $K$ of $G$ contains a subgroup $N$ of finite index in $K$ which is normal in $G$.
\end{itemize}
\end{lem}

\begin{proof}
$ $

$\boxed{\text{(i)}\Rightarrow\text{(ii)}}$ Let $K$ be an orbital closed subgroup of $G$. By Lemma \ref{lem: properties of isolator}(i), $\mathrm{i}_G(K)$ is an isolated orbital closed subgroup of $G$, and so (by assumption) is normal in $G$. Therefore, as it contains $K$ as a subgroup of finite index (by Proposition \ref{propn: orbital subgroups are open in their isolators}), it must contain each conjugate $K^g$ (for any $g\in G$) as a subgroup of finite index. But as $K$ is $G$-orbital, it only has finitely many $G$-conjugates, and so their intersection $K^\circ$ still has finite index in $\mathrm{i}_G(K)$ and hence also in $K$, and $K^\circ$ is normal in $G$ by construction.

$\boxed{\text{(ii)}\Rightarrow\text{(i)}}$ Let $H$ be an isolated orbital closed subgroup of $G$, and write $H^\circ$ for the largest normal subgroup of $G$ contained in $H$, which by (ii) must have finite index in $H$. Now clearly $H \leq \mathrm{i}_G(H^\circ)$ by definition of $\mathrm{i}_G(H^\circ)$, but also $\mathrm{i}_G(H^\circ) \leq H$ as $H$ is isolated and contains $H^\circ$. So $H$ is the $G$-isolator of a normal subgroup, and so by Lemma \ref{lem: properties of isolator}(i), $H$ is also normal in $G$.
\end{proof}

\newpage
\section{The Roseblade subgroup $\nio(G)$}

%
%
%
%

We begin this section by remarking that ``orbitally sound" is not too restrictive a condition. Recall:

\begin{lem}\label{lem: p-valued implies o.s.}
Let $G$ be a $p$-valuable group. Then $G$ is orbitally sound.\qed
\end{lem}

\begin{proof}
This is \cite[Proposition 5.9]{ardakovInv}, after remarking that the definitions of ``orbitally sound" given in Definition \ref{defn: orbitally sound} and in \cite[5.8]{ardakovInv} are equivalent by Lemma \ref{lem: defns of orbitally sound are equivalent}.
\end{proof}

The following two lemmas will allow us to find a large class of orbitally sound groups.

For the next lemma, fix the following notation. Let $G$ be a compact $p$-adic analytic group, and consider its $\mathbb{Q}_p$-Iwasawa algebra $\mathbb{Q}_pG := (\mathbb{Z}_p G)\Big[\frac{1}{p}\Big]$. Write $I$ for its augmentation ideal
$$I = \ker(\mathbb{Q}_pG \to \mathbb{Q}_p).$$
Recall that $I^k$ is generated over $\mathbb{Q}_pG$ by $\{(x_1 - 1)\dots (x_k - 1) \,\big|\, x_i\in G\}$. Now it is clear that $G$ acts unipotently on the series
$$\mathbb{Q}_pG > I > I^2 > I^3 > \dots,$$
i.e. for all $g\in G$, we have $(g-1)\mathbb{Q}_pG \subseteq I$ and $(g-1)I^k \subseteq I^{k+1}$.

Write also $\mathcal{U}_n$ for the subgroup of $GL_n(\mathbb{Q}_p)$ consisting of upper triangular unipotent matrices.

\begin{lem}\label{lem: dimension subgroups}
Write
$$D_k = \ker(G\to (\mathbb{Q}_pG/I^k)^\times),$$
the \emph{$k$-th rational dimension subgroup} of $G$, for all $k\geq 1$. Then the $D_k$ are a descending chain of isolated orbital closed normal subgroups of $G$. This chain eventually stabilises: that is, there exists some $t$ such that $D_n = D_t$ for all $n\geq t$.

Furthermore, if $G$ is torsion-free and nilpotent, $D_t = 1$, and $G$ is isomorphic to a closed subgroup of $\mathcal{U}_m$ for some $m$.
\end{lem}

\begin{proof}
By definition, it is clear that the $D_k$ are closed normal (hence orbital) subgroups of $G$; to show that they are isolated orbital, we will show that each $G/D_k$ is torsion-free.

Fix $k$. Consider the series of finite-dimensional $\mathbb{Q}_p$-vector spaces
$$\mathbb{Q}_pG/I^k > I/I^k > I^2/I^k > I^3/I^k > \dots > I^k/I^k,$$
and choose a basis for $\mathbb{Q}_pG/I^k$ which is filtered relative to this series: i.e. by repeatedly extending a basis for $I^r/I^k$ to a basis for $I^{r-1}/I^k$, we get a basis
$$B = \{e_1, \dots, e_r\}$$
and integers
$$0 = n_k < n_{k-1} < n_{k-2} < \dots < n_1 < n_0 = r$$
with the property that $\{e_1, \dots, e_{n_r}\}$ is a basis for $I^r/I^k$ for each $0\leq r\leq k$ (where we write $I^0 = \mathbb{Q}_pG$ for convenience).

As $G$ acts \emph{unipotently} (by left multiplication) on $\mathbb{Q}_pG/I^k$, and by definition of the basis $B$, we see that with respect to $B$, each $g\in G$ acts by a unipotent upper-triangular matrix, i.e. we get a continuous group homomorphism $G\to \mathcal{U}_r$. Now $D_k$ is just the kernel of this map; but $\mathcal{U}_r$ is torsion-free, so $D_k$ must be isolated.

Recall the \emph{dimension} $\dim H$ of a pro-$p$ group $H$ of finite rank from \cite[4.7]{DDMS}. As $G$ has finite rank, it also has finite dimension \cite[3.11, 3.12]{DDMS}, and we must have $\dim D_i \geq \dim D_{i+1}$ for all $i$ by \cite[4.8]{DDMS}. But if $\dim D_i = \dim D_{i+1}$, then $D_i/D_{i+1}$ is a $p$-valued group (as $D_{i+1}$ is isolated) of dimension $0$ (again by \cite[4.8]{DDMS}), and so must be trivial. Hence the sequence $(D_i)$ stabilises after at most $t := 1 + \dim G$ terms, and so $$D_t = \bigcap_{n\geq 1} D_n.$$

Now suppose that $G$ is nilpotent. Then, by \cite[Theorem A]{ardakovGMJ}, it follows that $I$ is localisable. Let $R = (\mathbb{Q}_pG)_I$ be its localisation, and $J(R) = \mathfrak{m}$ its unique maximal ideal: then the ideal $$A = \bigcap_{n\geq 1} \mathfrak{m}^n$$ satisfies $A = \mathfrak{m} A$, so by Nakayama's lemma \cite[0.3.10]{MR}, we must have $A = 0$. This implies that
$$\bigcap_{n\geq 1} I^n \subseteq \ker(\mathbb{Q}_pG\to R).$$
Assuming further that $G$ is torsion-free, we see that $\mathbb{Q}_pG$ is a domain \cite[Theorem 1]{neumann}, and so the localisation map $\mathbb{Q}_pG\to R$ is injective. Hence $\cap_{n\geq 1} I^n = 0$, and so $$D_t = \bigcap_{n\geq 1} D_n = \left(\bigcap_{n\geq 1} (I^n + 1)\cap G\right) \subseteq \left(\bigcap_{n\geq 1} I^n\right) + 1 = 1.$$
Now the representation $G\to \Aut(\mathbb{Q}_pG/I^t) \cong GL_m(\mathbb{Q}_p)$ is faithful and has image in $\mathcal{U}_m$.
\end{proof}

\begin{lem}\label{lem: wehrfritz}
Let $G$ be a (topologically) finitely generated nilpotent pro-$p$ group. Then $G$ is $p$-valuable if and only if it is torsion-free.
\end{lem}

\begin{proof}
Lemma \ref{lem: dimension subgroups} gives an injective map $G\to \mathcal{U}_m$. Now, as $G$ is topologically finitely generated, its image in $\mathcal{U}_m$ must lie inside the set $\frac{1}{p^t} M_m(\mathbb{Z}_p) \cap \mathcal{U}_m$ for some $t$. Hence, by conjugating by the diagonal element
$$\mathrm{diag}(p, p^2, \dots, p^m)^{t+\varepsilon} \in GL_m(\mathbb{Q}_p),$$
where
$$\varepsilon =
\begin{cases}
1 & p>2,\\
2 & p=2,
\end{cases}$$
we see that $G$ is isomorphic to a subgroup of
$$\Gamma_\varepsilon = \left\{\gamma\in GL_m(\mathbb{Z}_p) \,\big|\, \gamma \equiv 1 (\text{mod } p^\varepsilon)\right\},$$
the $\varepsilon $th congruence subgroup of $GL_m(\mathbb{Z}_p)$, which is uniform (and hence $p$-valuable) by \cite[Theorem 5.2]{DDMS}.

The reverse implication is clear from the definition of a $p$-valuation \cite[III, 2.1.2]{lazard}.\qedhere
\end{proof}

Now we have found a large class of orbitally sound compact $p$-adic analytic groups.

\begin{cor}\label{cor: f-by-n implies o.s.}
If $G$ is a finite-by-nilpotent compact $p$-adic analytic group, then it is orbitally sound.
\end{cor}

\begin{proof}
$\overline{G} := G/\Delta^+$ must be a nilpotent compact $p$-adic analytic group with $\Delta^+(\overline{G}) = 1$, and so $\overline{G}$ is torsion-free by \cite[5.2.7]{rob}. Now Lemma \ref{lem: wehrfritz} shows that $\overline{G}$ is $p$-valuable, and from Lemma \ref{lem: p-valued implies o.s.} we may deduce that $\overline{G}$ is orbitally sound. But now Lemma \ref{lem: properties of orbitally sound groups}(iii) implies that $G$ is orbitally sound.
\end{proof}

\begin{rk}
It is well known that finite-by-nilpotent implies nilpotent-by-finite, but not conversely. Not all nilpotent-by-finite compact $p$-adic analytic groups are orbitally sound: indeed, the wreath product
$$G = \mathbb{Z}_p \wr C_2 = (\mathbb{Z}_p \times \mathbb{Z}_p) \rtimes C_2$$
is \emph{abelian}-by-finite, and the infinite procyclic subgroup $H = \mathbb{Z}_p \times \{0\}$ is orbital, but the largest $G$-normal subgroup contained in $H$ is the trivial subgroup.
\end{rk}

We can now define the Roseblade subgroup.

\begin{defn}\label{defn: nio}
As in Roseblade \cite[1.3]{roseblade}, write $\mathrm{nio}(G)$ for the closed subgroup of $G$ defined by $$\mathrm{nio}(G) = \bigcap_{H} \mathbf{N}_G(H),$$ where the intersection ranges over the isolated orbital closed subgroups $H$ of $G$.
\end{defn}

\begin{thm}\label{thm: nio is os}
Let $G$ be a compact $p$-adic analytic group.

\begin{itemize}
\item[(i)] An orbitally sound open normal subgroup $N\lhd G$ normalises every closed isolated orbital subgroup $H\leq G$.
\item[(ii)] $\mathrm{nio}(G)$ is an orbitally sound open characteristic subgroup of $G$.
\end{itemize}
\end{thm}

\begin{proof}
$ $

\begin{itemize}
\item[(i)] Since $H$ is isolated orbital in $G$, we must have that $H\cap N$ is isolated orbital (and hence normal) in $N$. However, it follows from Proposition \ref{propn: orbital subgroups are open in their isolators} that $H = \mathrm{i}_G(H\cap N)$. Hence $H$ is generated by all the (finitely many) closed orbital subgroups $L_1, \dots, L_k$ of $G$ containing $H\cap N$ as a subgroup of finite index. Conjugation by $n\in N$ permutes these $L_i$, and so fixes $H$.
\item[(ii)] Let $N$ be a complete $p$-valued open normal subgroup of $G$ (e.g. by \cite[8.34]{DDMS}). Then \cite[Proposition 5.9]{ardakovInv} shows that $N$ is orbitally sound, and hence by (i) $N$ normalises all closed isolated orbital subgroups of $G$. So, by definition, $N \leq \mathrm{nio}(G)$, and so $[G: \mathrm{nio}(G)] \leq [G:N] < \infty$. Therefore $\mathrm{nio}(G)$ is open in $G$ as required. But by definition, $\mathrm{nio}(G)$ is the largest subgroup that normalises all isolated orbital subgroups of $G$, so by the correspondence of Lemma \ref{lem: 1-1 corresp between isolated orbitals} and Lemma \ref{lem: properties of isolator}(i), it normalises all isolated orbital subgroups of $\mathrm{nio}(G)$, i.e. it is orbitally sound.\qedhere
\end{itemize}
\end{proof}

\textit{Proof of Theorem A.} It is clear from Definition \ref{defn: nio} and Theorem \ref{thm: nio is os}(i), (ii) that $\nio(G)$ is the unique maximal \emph{orbitally sound} closed normal subgroup of $G$, and is hence characteristic in $G$. Corollary \ref{cor: f-by-n implies o.s.}, together with Fitting's theorem \cite[1B, Proposition 15]{segal-polycyclic}, implies that it contains all finite-by-nilpotent closed normal subgroups. \qed

%
%
%
%
%

\newpage
\section{$p$-saturations}

Recall \cite[III, 2.1.2]{lazard} that a \emph{$p$-valuation} of a group $G$ is a function $$\omega: G\to \Rinfty$$ satisfying the following properties:
\begin{itemize}
\item $\omega(x) = \infty$ if and only if $x = 1$,
\item $\omega(x) > (p-1)^{-1}$,
\item $\omega(x^{-1}y) \geq \inf\{\omega(x), \omega(y)\}$,
\item $\omega([x,y]) \geq \omega(x) + \omega(y)$,
\item $\omega(x^p) = \omega(x) + 1$,
\end{itemize}
for all $x, y\in G$. The group $G$, when endowed with the $p$-valuation $\omega$, is called \emph{$p$-valued}. On the other hand, a group $G$ is called \emph{$p$-valuable} \cite[III, 3.1.6]{lazard} if there exists a $p$-valuation $\omega$ of $G$ with respect to which $G$ is \emph{complete} of \emph{finite rank}.

Let $G$ be a $p$-valuable group, and fix a $p$-valuation $\omega$ on $G$, so that $G$ is complete $p$-valued of finite rank. Recall the definition of the \emph{$p$-saturation} $\Sat G$ of $G$ (with respect to $\omega$) from \cite[IV, 3.3.1.1]{lazard}: this is again a complete $p$-valued group of finite rank, and there is a natural isometry identifying $G$ with an open subgroup of $\Sat G$ \cite[IV, 3.3.2.1]{lazard}. We will prove a few basic facts about $p$-saturations.

Firstly, we will prove a basic relationship between isolators and $p$-saturations.

\begin{lem}\label{lem: isolators and saturations}
Let $G$ be a complete $p$-valued group of finite rank, and let $H$ be a closed normal (and hence orbital) subgroup of $G$. Then $\mathrm{i}_G(H) = \Sat H \cap G$ (considered as subgroups of $\Sat G$).
\end{lem}

\begin{proof}
$[\Sat H : H] < \infty$ by \cite[IV, 3.4.1]{lazard}, and $\Sat H$ is a closed normal subgroup of $\Sat G$ by \cite[IV, 3.3.3]{lazard}, so $S := \Sat H \cap G$ is a normal (and hence orbital) subgroup of $G$, and contains $H$ as a subgroup of finite index. Hence, by Definition \ref{defn: isolator}, $S$ is contained in $\mathrm{i}_G(H)$.

To show the reverse inclusion, we will consider the group $\mathrm{i}_G(H)/S$, which is a finite subgroup of $G/S$ (as it is a quotient of $\mathrm{i}_G(H)/H$, which is finite by Proposition \ref{propn: orbital subgroups are open in their isolators}). But $G/S$ is isomorphic to $G\Sat H / \Sat H$, a subgroup of the torsion-free group $\Sat G / \Sat H$ (see \cite[IV, 3.4.2]{lazard} or \cite[III, 3.3.2.4]{lazard}). In particular, $G/S$ has no non-trivial finite subgroups, so we must have $\mathrm{i}_G(H) = S$.
\end{proof}

\begin{rk}
Of course, $\mathrm{i}_G(H)$ is independent of the choice of $\omega$.
\end{rk}

\begin{lem}\label{lem: quotients and saturations}
Let $G$ be a complete $p$-valued group of finite rank, which we again identify with an open subgroup of its $p$-saturation $S$. Suppose $S'$ is a $p$-saturated closed normal subgroup of $S$, and set $G' = S'\cap G$. Then there is a natural isometry $S/S' \cong \Sat (G/G')$.
\end{lem}

\begin{proof}
We will show that $S/S'$ satisfies the universal property for $\Sat (G/G')$ \cite[IV, 3.3.2.4]{lazard}. Clearly we may regard $G/G' \cong GS'/S'$ as a subgroup of $S/S'$. Note also that $S/S'$ is $p$-saturated, by \cite[III, 3.3.2.4]{lazard}. Also, as $G'$ is open in $S'$ and $S'$ is $p$-saturated, we have that $S' = \Sat G'$.

Let $H$ be an arbitrary $p$-saturated group and $\varphi: G/G' \to H$ a homomorphism of $p$-valued groups. We must first construct a map $\psi: S/S' \to H$. To do this, we first compose $\varphi$ with the natural surjection $G\to G/G'$ to get a map $\alpha: G\to H$, which we may then extend uniquely to a map $\beta: S\to H$ using the universal property of $S = \Sat G$, so that $\alpha = \beta|_G$ and the following diagram commutes.

\centerline{
\xymatrix{
G\ar@{->}@/^16pt/[rr]^-\alpha\ar@{->>}[r]\ar@{^(->}[d]& G/G'\ar@{->}[r]_-{\varphi}&H\\
S\ar@{->}@/_8pt/[urr]_-\beta&&
}
}

Now we wish to show that $\beta$ descends to a map $S/S' \to H$. To do this, we first study the restriction of $\alpha$ to $G'$ and of $\beta$ to $S'$. The following diagram commutes:

\centerline{
\xymatrix{
G'\ar@{->}[r]^-{\alpha|_{G'}}\ar@{^(->}[d]& H\\
S'\ar@{->}[ur]_-{\beta|_{S'}}&&
}
}

and so, since $S' = \Sat G'$, $\beta|_{S'}$ must be the \emph{unique} extension of $\alpha|_{G'}$ to a map $S'\to H$, as $S' = \Sat G'$. But $\alpha$ factors through $G/G'$, i.e. $\alpha|_{G'}$ is the trivial homomorphism $G' \to H$, so it extends to the trivial homomorphism $S'\to H$. By uniqueness, we must have $S'\subseteq \ker\beta$. This shows that $\beta$ induces a map $\psi: S/S' \to H$.

Finally, suppose $\varphi:G/G'\to H$ has two distinct extensions $\psi_1, \psi_2: S/S' \to H$. Then we may compose them with the natural surjection $S\to S/S'$ to get two distinct maps $\beta_1, \beta_2:S \to H$. Their restrictions $\alpha_1, \alpha_2:G\to H$ to $G$ must therefore also be distinct, for if not, then the map $\alpha_1 = \alpha_2:G\to H$ has (at least) two distinct extensions to maps $S \to H$, contradicting the universal property of $S = \Sat G$. Finally, if $\alpha_1$ and $\alpha_2$ are distinct, then they descend to distinct maps $\varphi_1, \varphi_2: G/G'\to H$, contradicting our assumption. So the extension of $\varphi$ to $\psi$ is unique.
\end{proof}

\begin{rk}
Lemma \ref{lem: quotients and saturations} holds even if $G$ does not have finite rank, and hence is only closed (not necessarily open) in its $p$-saturation $S$.
\end{rk}

\begin{defn}\label{defn: group series}
Let $G$ be an arbitrary group. A \emph{central series} for $G$ is a sequence of subgroups
$$G = G_1 \rhd G_2 \rhd \dots \rhd G_n = 1$$
with the property that $[G, G_i] \leq G_{i+1}$ for each $i$. (For the purposes of this definition, $G_j$ is understood to mean $1$ if $j > n$, and $G$ if $j < 1$.)

We will say that a central series is \emph{strongly} central if also $[G_i, G_j] \leq G_{i+j}$ for all $i$ and $j$.

An \emph{abelian series} for $G$ a sequence of subgroups
$$G = G_1 \rhd G_2 \rhd \dots \rhd G_n = 1$$
with the property that $[G_i, G_i] \leq G_{i+1}$ for each $i$.

When $G$ is a topological group, we will insist further that all of the $G_i$ should be \emph{closed} subgroups of $G$.
\end{defn}

\begin{rk}
We will be working with nilpotent $p$-valuable groups $G$. It will be useful for us to define the \emph{isolated lower central series} of $G$, which will turn out to be the fastest descending central series of closed subgroups $$G = G_1 \rhd G_2 \rhd \dots \rhd G_r = 1$$ with the property that the successive quotients $G_i / G_{i+1}$ are torsion-free (and hence $p$-valuable, by \cite[IV, 3.4.2]{lazard}). We will also prove that the isolated lower central series is a \emph{strongly} central series. (We demonstrate an \emph{isolated derived series} for soluble $p$-valued groups at the same time.)
\end{rk}

\begin{lem}\label{lem: abelian quotients and saturations}
Let $G$ be a complete $p$-valued group of finite rank, and $G_1 \geq G_2$ closed normal subgroups of $G$ with $G_1/G_2$ an abelian pro-$p$ group (which is not necessarily $p$-valued). Let $S_i = \Sat G_i$ for $i = 1,2$. Then $S_1/S_2$ is abelian and torsion-free (and hence $p$-valued), and has the same rank as $G_1/G_2$ as a $\mathbb{Z}_p$-module.
\end{lem}

\begin{proof}
As $S_2$ is $p$-saturated, $S_1/S_2$ is torsion-free, and so $$G_1 / (S_2\cap G_1) \cong G_1 S_2/S_2 \leq S_1/S_2$$ is torsion-free. $G_1/G_2$ maps onto $G_1 / (S_2\cap G_1)$ with finite kernel (by Lemma \ref{lem: isolators and saturations} and Proposition \ref{propn: orbital subgroups are open in their isolators}, and the assumption that $G$ has finite rank), so $G_1 / (S_2\cap G_1)$ is abelian of the same $\mathbb{Z}_p$-rank as $G_1/G_2$. By Lemma \ref{lem: quotients and saturations}, $S_1/S_2$ is the $p$-saturation of $G_1/(S_2\cap G_1)$, so is still abelian of the same $\mathbb{Z}_p$-rank.
\end{proof}

Before proving the main result of this section, we first need a technical lemma.

\begin{lem}\label{lem: commutators and isolators}
Let $G$ be a complete $p$-valued group of finite rank, and let $H$ and $N$ be two closed normal subgroups. Then
$$[\mathrm{i}_G(H), \mathrm{i}_G(N)] \leq \mathrm{i}_G(\overline{[H,N]}).$$
\end{lem}

\begin{proof}
Write $L := \mathrm{i}_G(\overline{[H,N]})$. This is normal in $G$, as $G$ is orbitally sound \cite[5.9]{ardakovInv}, and the quotient $G/L$ is still $p$-valued as it is torsion-free \cite[IV, 3.4.2]{lazard}.

Suppose first that $L = 1$, so that $[H,N]$ = 1. Then, for any $h\in \mathrm{i}_G(H)$, there is some integer $n$ such that $h^{p^n}\in H$, so that $[g, h^{p^n}] = 1$ for all $g\in N$. But this is the same as saying that $h^{p^n} = (h^g)^{p^n}$; and, as $G$ is $p$-valued, \cite[III, 2.1.4]{lazard} implies that $h = h^g$. As $g$ and $h$ were arbitrary, we see that $[\mathrm{i}_G(H), N] = 1$. Repeat this argument for $N$ to show that $[\mathrm{i}_G(H), \mathrm{i}_G(N)] = 1$.

If $L\neq 1$, we may pass to $G/L$. Write $\pi: G\to G/L$ for the natural surjection, so that
\begin{align}\label{eqn: 123}
\pi\big([\mathrm{i}_G(H), \mathrm{i}_G(N)]\big) = [\pi(\mathrm{i}_G(H)), \pi(\mathrm{i}_G(N))].
\end{align}
Now, $\pi(H)$ is a closed orbital subgroup of $\pi(G)$, and $\pi(\mathrm{i}_G(H))$ is a closed orbital subgroup of $\pi(G)$ containing $\pi(H)$ as an open subgroup, so that
\begin{align*}
\mathrm{i}_{\pi(G)}(\pi(H)) \geq \pi(\mathrm{i}_G(H)),
\end{align*}
and similarly for $N$. Together with (\ref{eqn: 123}), this implies that
\begin{align*}
\pi\big([\mathrm{i}_G(H), \mathrm{i}_G(N)]\big) \leq [\mathrm{i}_{\pi(G)}(\pi(H)),\mathrm{i}_{\pi(G)}(\pi(N))].
\end{align*}
But the right-hand side is now equal to $\pi(1)$, by the previous case, which shows that $[\mathrm{i}_G(H), \mathrm{i}_G(N)]\leq L$ as required.
\end{proof}

\begin{cor}\label{cor: N has a nice central series}
Let $G$ be a $p$-valuable group. Define two series:

\centerline{
\begin{tabular}{r l}
$G_i = \mathrm{i}_G(\overline{\gamma_i})$, where&
$\begin{cases}
\gamma_1 = G,\\
\gamma_{i+1} = [\gamma_i, G]& \mbox{for } i\geq 1;
\end{cases}$
\\
\\
$G^{(i)} = \mathrm{i}_G(\overline{\mathcal{D}_i})$, where&
$\begin{cases}
\mathcal{D}_0 = G,\\
\mathcal{D}_{i+1} = [\mathcal{D}_i, \mathcal{D}_i] & \mbox{for } i\geq 0,
\end{cases}$
\end{tabular}
}

where the bars denote topological closure inside $G$. If $G$ is nilpotent, then $(G_i)$ is a strongly central series for $G$, i.e. a central series in which $[G_i, G_j] \leq G_{i+j}$. If $G$ is soluble, then $(G^{(i)})$ is an abelian series for $G$. The quotients $G_i/G_{i+1}$ and $G^{(i)} / G^{(i+1)}$ are torsion-free, and hence $p$-valuable.
\end{cor}

\begin{rk}
We prove this using $p$-saturations, but the resulting closed subgroups $G_i$ and $G^{(i)}$ are independent of the choice of $p$-valuation $\omega$ on $G$.

The series $(G^{(i)})$ above is a generalisation of the series studied in \cite[proof of lemma 2.2.1]{nelson}, there called $(G_i)$.
\end{rk}

\begin{proof}
Fix a $p$-valuation $\omega$ on $G$ throughout.

Firstly, we will show that $(G_i)$ is an abelian series. The claim that $(G^{(i)})$ is an abelian series will follow by an identical argument.

The (abstract) lower central series $(\gamma_i)$ is an abelian series for $G$ as an abstract group (i.e. the subgroups $\gamma_i$ are not necessarily closed in $G$), and so the series $(\overline{\gamma_i})$ is a series of \emph{closed} normal subgroups of $G$, which is still an abelian series by continuity. Now, applying Lemma \ref{lem: abelian quotients and saturations} shows that $(\Sat \overline{\gamma_i})$ is also an abelian series; and by Lemma \ref{lem: isolators and saturations}, we see that $G_i = \Sat \overline{\gamma_i} \cap G$ for each $i$, so that $(G_i)$ is an abelian series.

Secondly, we address the claim that the quotients $G_i/G_{i+1}$ are torsion-free and hence $p$-valuable: this follows from \cite[III, 3.1.7.6 / IV, 3.4.2]{lazard}, as the $G_{i+1}$ are isolated in $G$. The case of the quotients $G^{(i)}/G^{(i+1)}$ is again identical.

Thirdly, we must show that $G_{i-1}/G_i$ is central in $G/G_i$. Certainly $\gamma_{i-1}G_i/G_i$ is central in $G/G_i$, because $\gamma_i \leq G_i$, and so $$\overline{\gamma_{i-1}}G_i/G_i \leq Z(G/G_i)$$ by continuity. However, \cite[lemma 8.4(a)]{ardakovInv} says that $Z(G/G_i)$ is isolated in $G/G_i$, so by taking $(G/G_i)$-isolators of both sides, we must have $$\displaystyle \mathrm{i}_{G/G_i}\big(\overline{\gamma_{i-1}}G_i/G_i\big) \leq Z(G/G_i);$$ and the left-hand side is clearly equal to $G_{i-1}/G_i$ by Lemma \ref{lem: properties of orbitally sound groups}(i) and Definition \ref{defn: isolator}.

Finally, note that
$$[\gamma_i, \gamma_j] \leq \gamma_{i+j}$$
by \cite[5.1.11(i)]{rob}, and so by taking closures,
$$\overline{[\gamma_i, \gamma_j]} \leq \overline{\gamma_{i+j}}.$$
But $[\overline{\gamma_i}, \overline{\gamma_j}] \leq \overline{[\gamma_i, \gamma_j]}$, as the function $G\times G\to G$ given by $(a, b) \mapsto [a,b]$ is continuous. Hence
$$[\overline{\gamma_i}, \overline{\gamma_j}]  \leq \overline{\gamma_{i+j}},$$
which implies
$$\overline{[\overline{\gamma_i}, \overline{\gamma_j}]}  \leq \overline{\gamma_{i+j}},$$
and so, by Lemma \ref{lem: commutators and isolators}, we may take isolators to show that
$$[\mathrm{i}_G(\overline{\gamma_i}), \mathrm{i}_G(\overline{\gamma_j})] \leq \mathrm{i}_G\Big(\overline{[\overline{\gamma_i}, \overline{\gamma_j}]}\Big) \leq \mathrm{i}_G(\overline{\gamma_{i+j}}),$$
i.e. $[G_i, G_j] \leq G_{i+j}.$
\end{proof}

\begin{defn}\label{defn: ilcs}
When $G$ is a nilpotent (resp. soluble) $p$-valued group of finite rank, the series $(G_i)$ (resp. $(G^{(i)})$) defined in Corollary \ref{cor: N has a nice central series} is the \emph{isolated lower central series} (resp. \emph{isolated derived series}) of $G$.
\end{defn}

\textit{Proof of Theorem B.} This is the content of Corollary \ref{cor: N has a nice central series}.\qed

\newpage
\section{Conjugation action of $G$}

In this subsection, we will study how nilpotent-by-finite compact $p$-adic analytic groups $G$ act by conjugation on certain torsion-free abelian and nilpotent subquotients. First, we slightly extend the term ``orbitally sound".

\begin{defn}
Let $G$ and $H$ be profinite groups, and suppose $G$ acts (continuously) on $H$. Then $G$ permutes the closed subgroups of $H$. We say that the action of $G$ on $H$ is \emph{orbitally sound} if, for any closed subgroup $K$ of $H$ with finite $G$-orbit, there exists an open subgroup $K'$ of $K$ which is normalised by $G$.
\end{defn}

Recall the group of torsion units of $\mathbb{Z}_p$: $$t(\mathbb{Z}_p^\times) = \begin{cases}
\{\pm 1\} & p=2\\
\mathbb{F}_p^\times & p>2.
\end{cases}$$

\begin{lem}\label{lem: F acts by scalars}
Let $A$ be a free abelian pro-$p$ group of finite rank. Let $G$ be a profinite group acting orbitally soundly and by automorphisms of finite order on $M$. Then, for each $g\in G$, there exists $$\zeta = \zeta_g \in t(\mathbb{Z}_p^\times)$$ such that $g\cdot x = \zeta x$ for all $x\in A$. This is multiplicative in $G$, in the sense that $\zeta_g \zeta_h = \zeta_{gh}$ for all $g,h\in G$.
\end{lem}

\begin{proof}
Write $\varphi$ for the automorphism of $A$ given by conjugation by $g$. We may view $\varphi$ as an automorphism of the $\mathbb{Q}_p$-vector space $A_{\mathbb{Q}_p} := A\tensor{\mathbb{Z}_p} \mathbb{Q}_p$.

As the action of $G$ on $A$ is orbitally sound, in particular, we have $$\langle x\rangle \cap \langle \varphi(x) \rangle \neq \{0\}$$ (as $\mathbb{Q}_p$-vector subspaces) for every $x\in A_{\mathbb{Q}_p}$. But this just means that $x$ is an eigenvector of the linear map $\varphi$. If all elements of $A_{\mathbb{Q}_p}$ are eigenvectors of $\varphi$, then they must have a common eigenvalue, say $\zeta$. The statement that $G$ acts on $A$ by automorphisms of finite order means that the eigenvalue $\zeta$ for $x$ is of finite order, $\zeta\in t(\mathbb{Z}_p^\times)$.

Multiplicativity is clear from the fact that $(gh)\cdot x = g\cdot (h\cdot x)$ for all $g,h\in G$.
\end{proof}

\begin{rk}
Assume that $G$ is a nilpotent-by-finite, orbitally sound compact $p$-adic analytic group. In the case when $H$ is an open subgroup of $G$ containing $\Delta^+$, with the property that $N := H/\Delta^+$ is nilpotent $p$-valuable, we may consider the isolated lower central series of Corollary \ref{cor: N has a nice central series} for $N$:
$$N = N_1 \rhd N_2 \rhd \dots \rhd N_r = 1,$$
and take their preimages in $G$ to get a series of characteristic subgroups of $H$:
$$H = H_1 \rhd H_2 \rhd \dots \rhd H_r = \Delta^+,$$
with the property that each $A_i := H_i/H_{i+1}$ is a free abelian pro-$p$ group of finite rank.

$G$ clearly acts orbitally soundly on each $A_i$, as $G$ is itself orbitally sound. Furthermore, as $[H, H_i] \leq H_{i+1}$ for each $i$, we see that the action $G\to \Aut(A_i)$ contains the open subgroup $H$ in its kernel, and so $G$ acts by automorphisms of finite order. Thus we may apply Lemma \ref{lem: F acts by scalars} to see that $G$ acts on each $A_i$ via a homomorphism $\xi_i: G\to t(\mathbb{Z}_p^\times)$.

That is, given any $g\in G$ and $h\in H_i$, and writing $\zeta = \xi_i(g)$ and $a = nN_{i+1}\in A_i$, we have
$$(n^g)n^{-\zeta}\in N_{i+1},$$
or equivalently (still in multiplicative notation)
$$a^g = a^\zeta.$$
\end{rk}

We now show that the action of an automorphism of $G$ on the quotients $A_i$ is strongly controlled by its action on $A_1$. This is an important property that the isolated lower central series shares with the usual lower central series of abstract nilpotent groups; cf. \cite[5.2.5]{rob} and the surrounding discussion.

\begin{lem}\label{lem: zetas are not all 1}
Let $H$ be a finite-by-(nilpotent $p$-valuable) group, and continue to write $A_i := H_i/H_{i+1}$ as in the remark above. Let $\alpha$ be an automorphism of $H$ inducing multiplication by $\zeta_i\in t(\mathbb{Z}_p^\times)$ on each $A_i$. Then $\zeta_i = \zeta_1^i$ for each $i$.
\end{lem}

\begin{proof}
The map
\begin{align*}
A_1 \tensor{\mathbb{Z}_p} A_i &\to A_{i+1}\\
xH_2 \otimes yH_{i+1} &\mapsto [x,y]H_{i+2}
\end{align*}
is a $\mathbb{Z}_p\langle \alpha\rangle$-module homomorphism, and its image is open in $A_{i+1}$ (by definition of the isolated lower central series). Write $\zeta_1 = \zeta$, and proceed by induction on $i$: suppose that $\zeta_i = \zeta^i$. Now, for any positive integers $a$ and $b$, we have
\begin{align*}
[x^a, y^b]H_{i+2} = [x,y]^{ab} H_{i+2}
\end{align*}
by \cite[0.2(i), (ii)]{DDMS} and by using the fact that $[x,y]H_{i+2}$ is central in $H/H_{i+2}$. Hence, by continuity, this is true for any $a,b \in \mathbb{Z}_p$, and so
\begin{align*}
\alpha([x,y]H_{i+2}) & = [x^\zeta, y^{\zeta^i}]H_{i+2} \\
& = [x,y]^{\zeta^{i+1}}H_{i+2}.\qedhere
\end{align*}
\end{proof}

We deduce:

\begin{cor}\label{cor: defn of xi_1}
Let $G$ be a nilpotent-by-finite, orbitally sound compact $p$-adic analytic group, and $H$ an open normal subgroup of $G$ containing $\Delta^+$ such that $H/\Delta^+$ is nilpotent $p$-valuable. Then the conjugation action of $G$ on $H$ induces an action of $G$ on $H/H_2$ given by the map $\xi_1: G\to t(\mathbb{Z}_p^\times) \leq \Aut(H/H_2)$ defined above. Moreover, $H\leq \ker \xi_1$.\qed
\end{cor}

\begin{rk}
If $N = H/\Delta^+$ is $p$-saturable, we may take its corresponding Lie algebra $L$ by Lazard's isomorphism of categories \cite[IV, 3.2.6]{lazard}. As in \cite[proof of lemma 8.5]{ardakovInv}: using \cite[lemma 4.2]{ardakovInv} and the fact that the $N/N_i$ are torsion-free, we can pick an \emph{ordered basis} \cite[III, 2.2.4]{lazard} $B$ for $N$ which is \emph{filtered} relative to the filtration on $N$: that is,
\begin{align*}
B = \{n_1, n_2, \dots, n_e\},
\end{align*}
and there exists a filtration of sets
\begin{align*}
B = B_1 \supset B_2 \supset \dots \supset B_{r-1} \neq \emptyset
\end{align*}
such that $B_i$ is an ordered basis for $N_i$ for each $1\leq i\leq r-1$. We may order the elements so that, for some integers $1 = k_1 < k_2 < \dots < k_{r-1} < e$, we have $B_i = \{n_{k_i+1}, \dots, n_e\}$ for each $1\leq i\leq r-1$. Taking logarithms of these basis elements gives us a basis for $L$, and then Lemma \ref{lem: zetas are not all 1} implies that, with respect to this basis, the automorphism of $L$ induced by $\alpha$ has the special block lower triangular form
\begin{align*}
\begin{pmatrix}
\zeta I_{d_1} & 0 & 0 & \dots & 0 \\
* & \zeta^2 I_{d_2} & 0 & \dots & 0 \\
* & * & \zeta^3 I_{d_3} & \dots & 0 \\
\vdots & \vdots & \vdots & \ddots & \vdots \\
* & * & * & \dots & \zeta^{r-1} I_{d_{r-1}}
\end{pmatrix},
\end{align*}
where $d_i = \mathrm{rk}(L_i/L_{i+1}) = \mathrm{rk}(M_i)$ and $I$ denotes the identity matrix.
\end{rk}

\newpage
\section{The finite-by-(nilpotent $p$-valuable) radical}

Let $G$ be a nilpotent-by-finite compact $p$-adic analytic group. Consider the set
$$\mathcal{S}(G) = \left\{ H\nopen G \middle| H/\Delta^+(H) \mbox{ is nilpotent and } p \mbox{-valuable}\right\},$$
where ``$H\nopen G$" means ``$H$ is an open normal subgroup of $G$". $\mathcal{S}(G)$ is nonempty, as we can pick an open normal nilpotent uniform subgroup of $G$ by \cite[4.1]{DDMS}, and hence contains a maximal element. We will show that this maximal element is \emph{unique}, and we will call this element the \emph{finite-by-(nilpotent $p$-valuable) radical} of $G$, and once we have shown its uniqueness we will denote it by $\fn(G)$.

\begin{rk}
Once we have shown the existence and uniqueness of $\fn(G)$, it will be clear that it is a \emph{characteristic} open subgroup of $G$ (as automorphisms of $G$ leave $\mathcal{S}(G)$ invariant), and contained in $\nio(G)$ (by Corollary \ref{cor: f-by-n implies o.s.} and Theorem A).

The quotient group $$\nio(G)/\fn(G)$$ is isomorphic to a subgroup of $t(\mathbb{Z}_p^\times)$ by Corollary \ref{cor: defn of xi_1}. When $p > 2$, $t(\mathbb{Z}_p^\times)$ is a $p'$-group, and so $\fn(G)/\Delta^+$ is the unique Sylow pro-$p$ subgroup of $\nio(G)/\Delta^+$. (This fails for $p=2$: the ``2-adic dihedral group" $G = \mathbb{Z}_2 \rtimes C_2$ has $\Delta^+(G) = 1$, $\nio(G) = G$, and is its own Sylow 2-subgroup, but $\fn(G) = \mathbb{Z}_2$.)
\end{rk}

In looking for maximal elements $H$ of $\mathcal{S}(G)$, we may make an immediate simplification. By maximality, any such $H$ must have $\Delta^+(H) = \Delta^+$, i.e. maximal elements of $\mathcal{S}(G)$ are in one-to-one correspondence with maximal elements of 
$$\mathcal{S}'(G) = \left\{ H\nopen G \middle| \Delta^+\leq H,\, H/\Delta^+ \mbox{ is nilpotent and } p \mbox{-valuable}\right\},$$
and this set is clearly in order-preserving one-to-one correspondence with the set $\mathcal{S}(G/\Delta^+)$. Hence we may immediately assume without loss of generality that $\Delta^+ = 1$.

\begin{lem}\label{lem: there exists a large N}
Let $G$ be a nilpotent-by-finite compact $p$-adic analytic group with $\Delta^+ = 1$. Then
\begin{itemize}
\item[(i)] there exists a nilpotent uniform open normal subgroup $H$ of $G$ which contains $\Delta$,
\item[(ii)] any such $H$ satisfies the property that $Z(H) = \Delta$.
\end{itemize}
\end{lem}

\begin{proof}
First, suppose we are given a nilpotent uniform open normal subgroup $H$. Take $x\in \Delta$: $\mathbf{C}_G(x)$ is open in $G$ by definition, and so 
$$\mathbf{C}_H(x) = \mathbf{C}_G(x) \cap H$$
is open in $H$. Therefore, for any $h\in H$, we can find some integer $k$ such that $h^{p^k} \in \mathbf{C}_H(x)$. This means that $(x^{-1}hx)^{p^k} = h^{p^k}$, and so by \cite[III, 2.1.4]{lazard}, we may take $(p^k)$th roots inside $H$ to see that $x^{-1}hx = h$. In other words, $x\in \mathbf{C}_\Delta(H)$.

Now suppose further that $H$ contains $\Delta$. Then $x\in Z(H)$. In fact, as we have $Z(H) \leq \Delta(H)$ by definition and $\Delta(H) \subseteq \Delta$ by Lemma \ref{lem: Delta and Delta+}(ii), we see that $\Delta$ is all of the centre of $H$. This establishes (ii).

To prove (i), let $N$ be an open normal nilpotent uniform subgroup \cite[4.1]{DDMS} of $G$. Form $H = N\Delta$, again an open normal subgroup of $G$. The first paragraph above shows that $[N,\Delta] = 1$; we also know from Lemma \ref{lem: Delta and Delta+}(v) that $\Delta$ is abelian and $N$ is nilpotent. This forces $H$ to be nilpotent and open in $G$, and to contain $\Delta$ in its centre. 

It remains only to show that $H$ is uniform. As $H$ is nilpotent, its set $t(H)$ of torsion elements forms a normal subgroup \cite[5.2.7]{rob}, and if $t(H)$ is non-trivial then $t(H) \cap Z(H)$ must be non-trivial \cite[5.2.1]{rob}; but $Z(H) = \Delta$ is torsion-free by Lemma \ref{lem: Delta and Delta+}(v), so $H$ must be torsion-free. Now it is easy to check that $H$ is powerful as in \cite[3.1]{DDMS}, so that $H$ is uniform by \cite[4.5]{DDMS}.
\end{proof}

%
%

\begin{lem}\label{lem: S(G) has a unique max element}
Let $G$ be a nilpotent-by-finite compact $p$-adic analytic group. Then $\mathcal{S}(G)$ is closed under finite joins, and hence contains a \emph{unique} maximal element $H$, which is characteristic as a subgroup of $G$.
\end{lem}

\begin{proof}
First, observe that, for an open normal subgroup $K$ of $G$, we have $K \in \mathcal{S}(G)$ if and only if $\overline{K}\in \mathcal{S}(\overline{G})$ (where bars denote quotient by $\Delta^+$). So we continue to assume without loss of generality that $\Delta^+ = 1$.

Suppose we are given $K, L\in \mathcal{S}(G)$: then it remains to show that $KL\in \mathcal{S}(G)$. As $K$ and $L$ are open and normal, it is obvious that $KL$ is too; and since $K$ and $L$ are also nilpotent, Fitting's theorem \cite[1B, Proposition 15]{segal-polycyclic} implies that $KL$ is nilpotent. But now, again by \cite[5.2.7]{rob}, $t(KL) = \Delta^+(KL) \leq \Delta^+ = 1$ -- that is, $KL$ is torsion-free, and hence $p$-valuable by Lemma \ref{lem: wehrfritz}.

Now let $H$ be a maximal element of $\mathcal{S}(G)$. Assume for contradiction that $H$ does not contain every other element of $\mathcal{S}(G)$ as a subgroup. Then we may pick some $L\in\mathcal{S}(G)$ not contained in $H$, and form $HL\in \mathcal{S}(G)$; but now $H \lneq HL$, a contradiction to the maximality of $H$. So $H$ must be the \emph{unique} maximal element of $\mathcal{S}(G)$.

As the set $\mathcal{S}(G)$ is invariant under automorphisms of $G$, this maximal element $H$ is characteristic in $G$.
\end{proof}

\begin{defn}
Let $G$ be a nilpotent-by-finite compact $p$-adic analytic group. Its \emph{finite-by-(nilpotent $p$-valuable) radical} $\fn(G)$ is the open characteristic subgroup defined in Lemma \ref{lem: S(G) has a unique max element}.
\end{defn}

\textit{Proof of Theorem C.} The uniqueness of $H = \fn(G)$ is the content of Lemma \ref{lem: S(G) has a unique max element}. As $H$ is a finite-by-nilpotent closed subgroup of $G$, Corollary \ref{cor: f-by-n implies o.s.} and Theorem B show that $H$ is contained in $\nio(G)$. The inclusion $\Delta\leq H$ follows from Lemma \ref{lem: there exists a large N}(i). Finally, the isomorphism from $\nio(G)/H$ to a subgroup of $t(\mathbb{Z}_p^\times)$ is induced by the map $$\xi_1: \nio(G)\to t(\mathbb{Z}_p^\times)$$ of Corollary \ref{cor: defn of xi_1}, whose kernel contains $H$. \qed

\textit{Proof of Theorem D.} This now follows from Theorem C together with Lemma \ref{lem: zetas are not all 1}.\qed

\bibliography{biblio}
\bibliographystyle{plain}

\end{document}